\documentclass{amsart}

\usepackage{amsmath,amsthm,amssymb}
\usepackage{url}
\usepackage{xspace}
\usepackage[all,poly]{xy}
\usepackage{hyphenat}

\usepackage[backend=bibtex,maxbibnames=99,doi=false,isbn=false,sorting=nyt,style=alphabetic]{biblatex}
\DeclareFieldFormat[online]{title}{``#1"}
\DeclareFieldFormat[online]{eprint}{\nolinkurl{#1}}
\renewbibmacro{in:}{\ifentrytype{article}{}{\printtext{\bibstring{in}\intitlepunct}}}
\bibliography{masterbib}

\usepackage{hyperref}

\theoremstyle{plain}
\newtheorem{theorem}{Theorem}[section]
\newtheorem{lemma}[theorem]{Lemma}
\newtheorem{corollary}[theorem]{Corollary}
\newtheorem{proposition}[theorem]{Proposition}

\theoremstyle{definition}
\newtheorem{definition}[theorem]{Definition}
\newtheorem{remark}[theorem]{Remark}
\newtheorem{example}[theorem]{Example}
\newtheorem{question}[theorem]{Question}
\newtheorem{conjecture}[theorem]{Conjecture}
\newtheorem{thmdef}[theorem]{Theorem/Definition}

\newcommand{\E}[1]{\exists #1\,}
\newcommand{\A}[1]{\forall #1\,}
\newcommand{\ie}{i.e.\ }

\newcommand{\TRG}{T_\text{RG}}
\newcommand{\Kfeq}{K_{\text{feq}}}
\newcommand{\Tfeq}{T^*_{\text{feq}}}
\newcommand{\KCPZ}{K_{\text{CPZ}}}
\newcommand{\TCPZ}{T_{\text{CPZ}}}

\newcommand{\Fraisse}{Fra\"iss\'e\xspace}
\newcommand{\Erdos}{Erd\H{o}s\xspace}

\def\Ind#1#2{#1\setbox0=\hbox{$#1x$}\kern\wd0\hbox to 0pt{\hss$#1\mid$\hss}\lower.9\ht0\hbox to 0pt{\hss$#1\smile$\hss}\kern\wd0}
\def\Notind#1#2{#1\setbox0=\hbox{$#1x$}\kern\wd0\hbox to 0pt{\mathchardef
	\nn=12854\hss$#1\nn$\kern1.4\wd0\hss}\hbox to
	0pt{\hss$#1\mid$\hss}\lower.9\ht0 \hbox to
	0pt{\hss$#1\smile$\hss}\kern\wd0}
\newcommand{\ind}[1][]{\mathop{\mathpalette\Ind{}^{\!\!\!\!\rlap{$\scriptscriptstyle\textnormal{#1}$}\,\,\,\,}}}

\begin{document}

\title[Disjoint $n$-amalgamation]{Disjoint $n$-amalgamation and pseudofinite countably categorical theories}
\author{Alex Kruckman}

\begin{abstract}
Disjoint $n$-amalgamation is a condition on a complete first-order theory specifying that certain locally consistent families of types are also globally consistent. In this paper, we show that if a countably categorical theory $T$ admits an expansion with disjoint $n$-amalgamation for all $n$, then $T$ is pseudofinite. All theories which admit an expansion with disjoint $n$-amalgamation for all $n$ are simple, but the method can be extended, using filtrations of \Fraisse classes, to show that certain non-simple theories are pseudofinite. As case studies, we examine two generic theories of equivalence relations, $\Tfeq$ and $\TCPZ$, and show that both are pseudofinite. The theories $\Tfeq$ and $\TCPZ$ are not simple, but they have NSOP$_1$. This is established here for $\TCPZ$ for the first time.
\end{abstract}

\maketitle

\section{Introduction}\label{sec:intro}

The theory $\TRG$ of the random graph (also called the Rado graph) arises naturally in two distinct ways. First, the random graph is the \Fraisse limit of the class of all finite graphs $\mathcal{G}$: the unique countable ultrahomogeneous graph which embeds a copy of each finite graph. Second, $\TRG$ is the almost-sure theory of finite graphs, in the sense of zero-one laws: letting $\mathcal{G}(n)$ be the set of (labeled) graphs of size $n$ and $\mu_n$ the uniform measure on $\mathcal{G}(n)$, we have, for every sentence $\varphi$, $$\lim_{n\to\infty} \mu_n(\{G\in \mathcal{G}(n)\mid G\models \varphi\}) = 1 \iff \varphi\in \TRG.$$

The latter observation shows that $\TRG$ is pseudofinite; that is, every sentence in the theory has a finite model. In fact, the probabilistic argument shows that each sentence $\varphi\in \TRG$ has many finite models. For large $n$, most finite graphs of size $n$ satisfy $\varphi$. 

The situation is very different for the class $\mathcal{G}_\triangle$ of finite triangle-free graphs. $\mathcal{G}_\triangle$ has a \Fraisse limit, the generic triangle-free graph~\cite{Henson} (also called the Henson graph). $\mathcal{G}_\triangle$ also has a zero-one law for the uniform measures $\mu_n$ on $\mathcal{G}_\triangle(n)$, but its almost-sure theory diverges from its generic theory. Indeed, \Erdos, Kleitman, and Rothschild~\cite{EKR} showed that almost all large finite triangle-free graphs are bipartite, and hence do not contain any cycles of odd length, in contrast to the generic triangle-free graph. 

So the probabilistic argument that showed that the theory of the random graph is pseudofinite fails for the generic triangle-free graph. In fact, it is still unknown whether the theory of the generic triangle-free graph is pseudofinite (see~\cite{Cherlin, Cherlin2}). This state of affairs suggests the following very general question:

\begin{question}\label{question:q1}
When does a \Fraisse limit have a pseudofinite theory?
\end{question}

There are, essentially, two ways to show that a theory $T$ is pseudofinite. The first way is to construct finite structures which satisfy arbitrary finite subsets of $T$. An example in the case of the random graph is the sequence of Paley graphs: For each prime power $q \equiv 1 (\mathrm{mod}\, 4)$, define a graph with domain the finite field $\mathbb{F}_q$, putting an edge between distinct elements $a$ and $b$ just in case $a-b$ is a square in $\mathbb{F}_q$. Then the theories of the Paley graphs converge to $\TRG$ (see \cite{BEH} for details, and \cite{BR} for other explicit constructions).

The second way is via a probabilistic argument. Usually, this amounts to specifying a probability measure $\mu_n$ on some class $K(n)$ of finite $L$-structures for all $n\in \omega$, such that $$\lim_{n\to\infty} \mu_n(\{A\in K(n)\mid A\models \varphi\}) = 1 \iff \varphi\in T.$$ 

The first method has the advantage of being more explicit, and the constructions may be of combinatorial interest. But the second method tells us something more, assuming that the measures $\mu_n$ are natural enough: not only do the sentences of $T$ have finite models, but \emph{most} large structures in some class satisfy the sentences in $T$. Of course, the meaning of ``natural" is left intentionally vague. For example, the measure $\mu_n$ should not concentrate on the $n^\mathrm{th}$ element of some explicit sequence! Refining our question,

\begin{question}\label{question:q2}
When does a \Fraisse limit have a pseudofinite theory for a good probabilistic reason? For example, when is it the almost-sure theory for a natural sequence $(K(n),\mu_n)_{n\in\omega}$ of classes of finite structures equipped with probability measures?
\end{question}

An example of a \Fraisse limit which is pseudofinite, but \emph{not} for a good probabilistic reason, is the vector space $V$ of countably infinite dimension over a finite field. The finite models of sentences in $\mathrm{Th}(V)$ are few and far between, existing only in certain finite sizes and unique up to isomorphism in those sizes. This is one of a whole family of examples of a similar character, the smoothly approximable structures, studied by Kantor, Liebeck, and Macpherson in \cite{KLM} and classified by Cherlin and Hrushovski in \cite{CHBook}. Smoothly approximable structures are essentially algebraic: they are coordinatized by certain geometries coming from vector spaces equipped with bilinear forms.

The main purpose of this paper is advance a claim that ``combinatorial" \Fraisse limits (in contrast to the algebraic smoothly approximable structures) which are pseudofinite tend to be pseudofinite for a good probabilistic reason, and moreover that this good probabilistic reason tends to rely on a combinatorial condition, disjoint $n$-amalgamation, which generalizes the disjoint (or ``strong") amalgamation property for \Fraisse classes with trivial algebraic closure. 

The starting point is Theorem~\ref{thm:main}, which shows that for a countably categorical theory, disjoint $n$-amalgamation for all $n$ is a sufficient condition for pseudofiniteness. The hypothesis of disjoint $n$-amalgamation for all $n$ is very strong; however, almost all examples of ``combinatorial" countably categorical theories which are known to be pseudofinite either have disjoint $n$-amalgamation for all $n$, or are reducts of theories with disjoint $n$-amalgamation for all $n$. Such theories are simple (in the sense of the model-theoretic dividing line), see Theorem~\ref{thm:simple}. The only exceptions (that I am aware of at the time of this writing) are built from equivalence relations. 

In \cite{KP}, Kim and Pillay made the ``rather outrageous conjecture" that every pseudofinite countably categorical theory is simple. The generic theory of a parameterized family of equivalence relations, $\Tfeq$, was suggested by Shelah as a counterexample to this conjecture. However, to my knowledge, no proof that $\Tfeq$ is pseudofinite has appeared in the literature.

In this paper, I demonstrate pseudofiniteness of $\Tfeq$ (Section~\ref{sec:Tfeq}), as well as another generic theory of equivalence relations, $\TCPZ$ (Section~\ref{sec:TCPZ}), which was introduced (and shown to not be simple) by Casanovas, Pel\'{a}ez, and Ziegler \cite{CPZ}. In both cases, the argument relies on a method of filtering the relevant \Fraisse class as a union of simpler \Fraisse classes, each of which admits an expansion to a countably categorical theory with disjoint $n$-amalgamation for all $n$. This shows that the pseudofiniteness of these examples, too, can be viewed as a consequence of a probabilistic argument involving disjoint $n$-amalgamation. An interesting feature of this method is that each sentence of the theory is shown to be in the almost-sure theory for a sequence $(K(n),\mu_n)_{n\in\omega}$ of classes of finite structures equipped with probability measures, and hence is pseudofinite for a good probabilistic reason, but different sentences require different sequences.

Countably categorical pseudofinite theories do not have the strict order property. Since I am not aware of a reference for this folklore result, I will give a proof here:

\begin{proposition}\label{prop:nsop}
No countably categorical categorical pseudofinite theory has the strict order property.
\end{proposition}
\begin{proof}
If $T$ has the strict order property, then it interprets a partial order with infinite chains. So it suffices to show that no countably categorical partial order $(P,<)$ with infinite chains is pseudofinite.

By compactness, we can find an infinite increasing chain $\{a_i \mid i \in \omega\}$ with $P\models a_i < a_j$ if and only if $i<j$. In a countably categorical theory, automorphism-invariant properties are definable, so there is a formula $\varphi(x)$, with $\varphi(x)\in \mathrm{tp}(a_i)$ for all $i$, such that $P\models \varphi(b)$ if and only if there is an infinite increasing chain above $b$.

Now $P\models \E{x} \varphi(x) \land \A{x} (\varphi(x)\rightarrow \E{y} (x < y \land \varphi(y)))$. But in any partial order, this sentence implies the existence of an infinite increasing chain of elements satisfying $\varphi(x)$, so its conjunction with the partial order axioms has no finite model.
\end{proof}

In \cite{DS}, D\v{z}amonja and Shelah introduced the property SOP$_1$. It is the first in a linearly ordered hierarchy of combinatorial properties called SOP$_n$ (for $n$-Strong Order Property), which were originally defined by Shelah for $n\geq 3$ in \cite{Shelah500}. A theory has NSOP$_n$ if it does Not have the $n$-Strong Order Property. As usual in model theory, the named properties are bad: theories with NSOP$_n$ are tamer than theories with SOP$_n$. These properties lie strictly between non-simplicity and the Strict Order Property (SOP): \[\text{simple} \implies \mathrm{NSOP}_1 \implies \dots \implies \mathrm{NSOP}_n \implies \dots \implies \mathrm{NSOP}.\] It is worth noting that SOP$_2$ also goes by the name TP$_1$ (the Tree Property of the first kind, see~\cite{KimKim} for a discussion), and every theory which is known to have SOP$_1$ also has SOP$_3$. So it is possible that $\mathrm{NSOP}_1 \overset{?}{=} \mathrm{NSOP}_2 = \mathrm{NTP}_1 \overset{?}{=} \mathrm{NSOP}_3$. The generic triangle-free graph has SOP$_3$ but NSOP$_4$ \cite{Shelah500}. 

Chernikov and Ramsey \cite{CR} gave a independence relation criterion for NSOP$_1$ and used it to show that $\Tfeq$ has NSOP$_1$. The theory $\TCPZ$ was not considered in \cite{CR}, but the methods there also suffice to show that $\TCPZ$ has NSOP$_1$ (see Corollary~\ref{cor:CPZ}). On the other hand, almost nothing is known about pseudofiniteness of countably categorical theories in the region between SOP$_1$ and the strict order property. While acknowledging that we have a paucity of other examples, I think it is reasonable to update the outrageous conjecture of Kim and Pillay in the following way:

\begin{conjecture}\label{conj:outrageous}
Every pseudofinite countably categorical theory has NSOP$_1$. 
\end{conjecture}

In Section~\ref{sec:prelim}, I review the relevant background on \Fraisse theory. I introduce disjoint $n$-amalgamation in Section~\ref{sec:defs} and prove Theorem~\ref{thm:main} in Section~\ref{sec:pseudofinite}. Section~\ref{sec:notions} contains some context about the role of $n$-amalgamation properties in model theory, as well as an explanation of how Theorem~\ref{thm:main} generalizes and unifies previous work. In Section~\ref{sec:eqrel}, I introduce the notion of a filtered \Fraisse class and give the applications to generic theories of equivalence relations ($\Tfeq$ and $\TCPZ$), along with a negative result, Proposition~\ref{prop:nohenson}, showing that this method cannot be used to show that the generic triangle-free graph is pseudofinite.

\bigskip

\textbf{Acknowledgements:} I would like to thank Tom Scanlon for his support and for many helpful discussions. Nick Ramsey also had a great influence on this paper: he told me about the theory $\TCPZ$, suggested Conjecture~\ref{conj:outrageous}, and, after some effort, convinced me that SOP$_1$ is a natural dividing line.

\section{Preliminaries}\label{sec:prelim}

In this section, I give a brief review of \Fraisse theory. The ``canonical language" described in Definition~\ref{def:canonical} provides the bridge to general countably categorical theories. For proofs, see \cite[Sections 2.6--8]{Cameron} or \cite[Section 7.1]{Hodges}.

Let $L$ be a relational language (not necessarily finite), and let $K$ be a class of finite $L$-structures which is closed under isomorphism. 
\begin{itemize}
\item $K$ has the \emph{hereditary property} if it is closed under substructure.
\item $K$ has the \emph{joint embedding property} if for all $A,B\in K$, there exists $C\in K$ and embeddings $A\hookrightarrow C$ and $B\hookrightarrow C$.
\item $K$ has the \emph{amalgamation property} (or \emph{$2$-amalgamation}) if for all $A,B,C\in K$ and embeddings $f\colon A\hookrightarrow B$ and $g\colon A\hookrightarrow C$, there exists $D\in K$ and embeddings $f'\colon B\hookrightarrow D$ and $g'\colon C\hookrightarrow D$ such that $f'\circ f = g'\circ g$. 
\item $K$ has the \emph{disjoint amalgamation property} (or \emph{disjoint $2$-amalgamation}) if, in the definition of the amalgamation property, the images of $B$ and $C$ in $D$ can additionally be taken to be disjoint over the image of $A$ in $D$: $(f'\circ f)[A] = (g'\circ g)[A] = f'[B] \cap g'[C]$. 
\item $K$ is a \emph{weak \Fraisse class} if it is countable up to isomorphism and has the hereditary property, the joint embedding property, and the amalgamation property.
\item $K$ is a \emph{\Fraisse class} if it is a weak \Fraisse class and additionally $K$ contains only finitely many structures of size $n$ up to isomorphism for all $n\in\omega$.
\end{itemize}

\begin{remark}\label{rem:terminology}
What I call a weak \Fraisse class here is often simply called a \Fraisse class. However, as we are only interested in \Fraisse classes with countably categorical generic theory, it is convenient to include the finiteness condition in the definition. Note that in a finite relational language, the notions coincide. In many sources the disjoint amalgamation property is called the strong amalgamation property. 
\end{remark}

\begin{definition}\label{def:ctbledefs}
Let $M$ be a countable $L$-structure.
\begin{itemize}
\item The \emph{age} of $M$ is the class of all finite structures which embed in $M$.
\item $M$ is \emph{ultrahomogeneous} if every isomorphism between finite substructures of $M$ extends to an automorphism of $M$.
\item $M$ has \emph{trivial acl} if $\mathrm{acl}(A) = A$ for all $A\subseteq M$.
\end{itemize}
\end{definition}

\begin{thmdef}\label{thmdef:fraisse}
$K$ is a weak \Fraisse class if and only if there is a countable ultrahomogeneous structure $M_K$ with age $K$. In this case, $M_K$ is unique up to isomorphism and is called the \emph{\Fraisse limit} of $K$. We call $T_K = \mathrm{Th}(M_K)$ the \emph{generic theory} of $K$. $K$ is a (strong) \Fraisse class if and only if $T_K$ is countably categorical. In this case, $T_K$ has quantifier elimination. A \Fraisse class $K$ has the disjoint amalgamation property if and only if $M_K$ has trivial acl. 
\end{thmdef}

Given a \Fraisse class $K$ and $n\in\omega$, I will write $K(n)$ for the (finite) set of structures in $K$ with domain $[n] = \{1,\dots,n\}$. Note that I include the empty structure in the case $n = 0$. $K(n)$ contains $(n!/|\mathrm{Aut}(A)|)$-many isomorphic copies of every structure $A$ in $K$. It will be convenient to identify these structures with their quantifier-free $n$-types: for $A\in K(n)$, $$\mathrm{qftp}(A) = \{\varphi(x_1,\dots,x_n)\mid \varphi\text{ is quantifier-free, and }A\models \varphi(1,\dots,n)\}.$$
Since $T_K$ has quantifier elimination, we can further identify the structures in $K(n)$ with the set of first-order $n$-types over the empty set relative to $T_K$ which are non-redundant, in the sense that they contain the formulas $\{x_i\neq x_j\mid i\neq j\}$.

Now each $n$-type relative to $T_K$ is isolated by a quantifier-free formula. In other words, each structure $A\in K(n)$ is distinguished from the others by a single quantifier-free formula $\theta_A(x_1,\dots,x_n)$. In the case that $L$ is finite, we may take $\theta_A$ to be the conjunction of the atomic diagram of $A$. If $L$ is infinite, a large enough part of the atomic diagram suffices.

\begin{theorem}\label{thm:axiomatization}
In this notation, the generic theory $T_K$ can be explicitly axiomatized as follows:
\begin{itemize}
\item The universal theory of $K$: This amounts to the sentences, for $n\in\omega$, $$\A{x_1,\dots,x_n} \left(\left(\bigwedge_{i\neq j} x_i\neq x_j\right) \rightarrow \left(\bigvee_{A\in K(n)} \theta_A(\overline{x})\right)\right),$$
together with, if $L$ is infinite, the information about how $\theta_A$ determines the other quantifier-free formulas. That is, for each $n$, $A\in K(n)$, and quantifier-free formula $\varphi(\overline{x})\in \mathrm{qftp}(A)$, $$\A{\overline{x}} \left(\theta_A(\overline{x})\rightarrow \varphi(\overline{x})\right).$$
\item One-point extension axioms: For all $A\in K(n)$ and $B\in K(n+1)$, we say that $(A,B)$ is a one-point extension if $A$ is the induced substructure of $B$ with domain $[n]$. Given a one-point extension $(A,B)$, we have the axiom $$\A{\overline{x}} \E{y} \left(\theta_A(\overline{x}) \rightarrow \theta_B(\overline{x},y)\right).$$
\end{itemize}
\end{theorem}

\begin{definition}\label{def:TKn}
$T_{K,n}$ is the (incomplete) theory axiomatized by 
\begin{enumerate}
\item The sentences in the universal theory of $K$ in at most $n$ universal quantifiers.
\item All one-point extension axioms for $K$ (with no restriction on the sizes of $A$ and $B$).
\end{enumerate}
\end{definition}

A model of $T_{K,n}$ satisfies all the one-point extension axioms over substructures satisfying one of the formulas $\theta_A$ for $A\in K$, but its age need only agree with $K$ up to substructures of size at most $n$. We will see in Theorem~\ref{thm:main} below that basic disjoint amalgamation up to level $n$ implies pseudofiniteness of $T_{K,n}$. 

It will be useful to consider expansions of $T_K$ at the level of the \Fraisse class $K$.

\begin{definition}\label{def:expansion}
Let $K$ and $K'$ be \Fraisse classes in languages $L$ and $L'$, respectively, such that $L\subseteq L'$. We say that $K'$ is a \emph{\Fraisse expansion} of $K$ if 
\begin{enumerate}
\item $K = \{A\restriction L\mid A\in K'\}$
\item For all one-point extensions $(A,B)$ in $K$, and every expansion of $A$ to a structure $A'$ in $K'$, there is an expansion of $B$ to a structure $B'$ in $K'$ such that $(A',B')$ is a one-point extension in $K'$.
\end{enumerate}
\end{definition}

\begin{theorem}\label{thm:expansion}
$K'$ is a \Fraisse expansion of $K$ if and only if the \Fraisse limit $M_{K'}$ of $K'$ is an expansion of the \Fraisse limit $M_K$ of $K$.
\end{theorem}
\begin{proof}
Suppose that $M_{K'}\restriction L = M_K$. Then $K = \mathrm{Age}(M_K) = \{A\restriction L\mid A\in \mathrm{Age}(M_{K'})\}$, and $\mathrm{Age}(M_{K'}) = K'$. Given a one-point extension $(A,B)$ and an expansion $A'$ of $A$, we can find a substructure of $M_{K'}$ isomorphic to $A'$. In the reduct, this substructure is isomorphic to $A$, and, since the one-point extension axiom for $(A,B)$ is true of $M_K$, it extends to a copy of $B$. We can take $B'$ to be the $L'$-structure on this subset of $M_{K'}$.

Conversely, to show that $M_{K'}$ is an expansion of $M_K$, by countable categoricity it suffices to show that $M_{K'}\restriction L$ satisfies the theory $T_K$. It clearly satisfies the universal part, since $\textrm{Age}(M_{K'}\restriction L) = \{A\restriction L\mid A\in K'\} = K$. For the extension axioms, suppose $(A,B)$ is a one-point extension, and we have a copy of $A$ in $M_{K'}\restriction L$. Let $A'$ be the $L'$-structure on this subset of $M_{K'}$. Since $K'$ is a \Fraisse expansion of $K$, we can find an expansion $B'$ of $B$ in $K'$ such that $(A',B')$ is a one-point extension, and, since the one-point extension axiom for $(A',B')$ is true of $M_{K'}$, our copy of $A'$ extends to a copy of $B'$. Hence, in the reduct, our copy of $A$ extends to a copy of $B$.
\end{proof}

\begin{definition}\label{def:canonical}
Let $T$ be any countably categorical $L$-theory, and let $M$ be its unique countable model. The \emph{canonical language} for $T$ is the language $L'$ with one $n$-ary relation symbol $R_p$ for each $n$-type $p(\overline{x})$ realized in $M$. 
\end{definition}

We make $M$ into an $L'$-structure $M'$ in the natural way by setting $M'\models R_p(\overline{a})$ if and only if $\overline{a}$ realizes $p(\overline{x})$ in $M$. Let $T' = \mathrm{Th}_{L'}(M')$. Then $T$ and $T'$ are interdefinable theories, $M'$ is ultrahomogeneous, and hence is the \Fraisse limit of its age $K$, and $K$ has the disjoint amalgamation property if and only if $M$ has trivial acl. Note that for each $A\in K(n)$, we may take the isolating formula $\theta_A$ to be one of the basic $n$-ary relation symbols $R_p$.

\section{Disjoint \texorpdfstring{$n$}{n}-amalgamation}\label{sec:namalg}

\subsection{Definitions}\label{sec:defs}

To fix notation, $[n] = \{1,\dots,n\}$, $\mathcal{P}([n])$ is the powerset of $[n]$, and $\mathcal{P}^-([n])$ is the set of all proper subsets of $[n]$. A family $\mathcal{F}\subseteq \mathcal{P}([n])$ of subsets of $[n]$ is \emph{downwards closed} if $S'\in\mathcal{F}$ whenever $S'\subseteq S$ and $S\in \mathcal{F}$. 

Let $T$ be a theory and $A$ a set of parameters in a model of $T$. We say that a type $p(\overline{x})$ over $A$ in the variables $\{x_i\mid i\in I\}$ is \emph{non-redundant} if it contains the formulas $\{x_i\neq x_j\mid i\neq j\in I\}$ and $\{x_i\neq a\mid i\in I, a\in A\}$. Given a downwards closed family of subsets $\mathcal{F}\subseteq \mathcal{P}([n])$, and variables $\overline{x}_1,\dots,\overline{x}_n$, a \emph{coherent $\mathcal{F}$-family of types over $A$} is a set $\{p_{S}\mid S\in\mathcal{F}\}$ such that each $p_{S}$ is a non-redundant type over $A$ in the variables $\overline{x}_S = \{\overline{x}_i\mid i\in S\}$, and $p_{S'}\subseteq p_S$ when $S'\subseteq S$. Here each $\overline{x}_i$ is a tuple of variables, possible empty or infinite, but such that $\overline{x}_i$ is disjoint from $\overline{x}_j$ when $i\neq j$.

For $n\geq 2$, a \emph{disjoint $n$-amalgamation problem} is a coherent $\mathcal{P}^-([n])$-family of types over a set $A$. A \emph{basic} disjoint $n$-amalgamation problem is a disjoint $n$-amalgamation problem over the empty set in the singleton variables $x_1,\dots,x_n$. 

A \emph{solution} to a (basic) disjoint $n$-amalgamation problem is an extension of the coherent $\mathcal{P}^-([n])$-family of types to a coherent $\mathcal{P}([n])$-family of types; that is, a non-redundant type $p_{[n]}$ such that $p_S\subseteq p_{[n]}$ for all $S$. We say $T$ has \emph{(basic) disjoint $n$-amalgamation} if every (basic) $n$-amalgamation problem has a solution.

If we replace $\mathcal{P}^-([n])$ by another downwards closed family of subsets $\mathcal{F}$ in the definitions above, we call the amalgamation problem \emph{partial}. 

First, some remarks on the definitions:

\begin{remark}\label{rem:empty}
In any coherent $\mathcal{F}$-family of types over $A$, the type $p_\varnothing$ is a $0$-type in the empty tuple of variables, which simply specifies the elementary diagram of the parameters $A$.
\end{remark}

\begin{remark}\label{rem:toplevel}
To specify a disjoint $n$-amalgamation problem, it would be sufficient to give the types $p_S$ for all $S$ with $|S| = n - 1$ and check that they agree on intersections, in the sense that $p_S\restriction \overline{x}_{S\cap S'} = p_S'\restriction \overline{x}_{S\cap S'}$ for all $S$ and $S'$. However, it is sometimes notationally convenient to keep the intermediate stages around.
\end{remark}

\begin{remark}\label{rem:fraisseamalg}
A \Fraisse class $K$ has the disjoint amalgamation property if and only if $T_K$ has disjoint $2$-amalgamation. Indeed, given $A,B,C\in K$ and embeddings $f\colon A\hookrightarrow B$ and $g\colon A\hookrightarrow C$, we take $A$ to be the base set of parameters, so $p_\varnothing = \mathrm{qftp}(A)$, and we set $p_{\{1\}}(\overline{x}_1) = \mathrm{qftp}((B\setminus A)/A)$ and $p_{\{2\}}(\overline{x}_2) = \mathrm{qftp}((C\setminus A)/A)$, identifying $A$ with its images in $B$ and $C$ under $f$ and $g$. By quantifier elimination, these quantifier-free types determine complete types relative to $T_K$. A solution to this disjoint $2$-amalgamation problem is the same as a structure $D$ in $K$ into which $B$ and $C$ embed disjointly over the image of $A$.
\end{remark}

\begin{remark}\label{rem:fraissenotation}
Given a \Fraisse class $K$, recall that we have identified $K(n)$, the structures in $K$ with domain $[n]$, with the set of non-redundant quantifier-free $n$-types relative to $T_K$. A basic disjoint $n$-amalgamation problem relative to $T_K$ is a coherent $\mathcal{P}^-([n])$-family of quantifier-free types $P = \{p_S\mid S\in\mathcal{P}^-([n])\}$ in the variables $x_1,\dots,x_{n}$, where each type $p_S$ corresponds to a structure $A_S$ in $K$ of size $|S|$. We write $K(n,P) = \{p_{[n]}(x_1,\dots,x_{n})\in K(n)\mid p_S\subseteq p_{[n]}\text{ for all }S\in \mathcal{P}^-([n])\}$ for the set of solutions to the amalgamation problem P, each of which corresponds to a structure $A_{[n]}$ in $K$ of size $n$ which contains all the $A_S$ as substructures. To say that $T_K$ has basic disjoint $n$-amalgamation is to say that $K(n,P)$ is non-empty for all $P$.
\end{remark}

It will be useful to observe that disjoint amalgamation gives solutions to partial amalgamation problems as well.

\begin{lemma}\label{lem:partial}
Suppose that $T$ has (basic) disjoint $k$-amalgamation for all $2\leq k\leq n$. Then every partial (basic) disjoint $n$-amalgamation problem has a solution.
\end{lemma}
\begin{proof}
I will consider the general case. The same proof works in the basic case.

We are given a partial disjoint $n$-amalgamation problem over $A$ in variables $\overline{x}_1,\dots,\overline{x}_n$; that is, a coherent $\mathcal{F}$-family of types $\{p_S\mid S\in\mathcal{F}\}$, with $\mathcal{F}\subseteq \mathcal{P}^-([n])$ downwards closed. 

We build a solution to the partial disjoint $n$-amalgamation problem from the bottom up. By induction on $1\leq k\leq n$, I claim that we can extend this family to a coherent $\mathcal{F}_k$-family of types, where $\mathcal{F}_k = \mathcal{F}\cup \{S\subseteq [n]\mid |S|\leq k\}$. When $k = n$, we have a coherent $\mathcal{P}([n])$-family of types, as desired.

When $k = 1$, if there is any $i$ such that $i\notin S$ for all $S\in \mathcal{F}$, then the original $\mathcal{F}$-family of types says nothing about the variables $\overline{x}_i$. We add $\{i\}$ into $\mathcal{F}_1$ and choose any non-redundant type $p_{\{i\}}$ over $A$ in the variables $\overline{x}_i$. If $\varnothing\notin \mathcal{F}$ (which only happens if $\mathcal{F}$ is empty) we also add it into $\mathcal{F}_1$, along with the unique $0$-type $p_\varnothing$ containing the elementary diagram of $A$.

Given a coherent $\mathcal{F}_{k-1}$-family of types by induction, with $2\leq k \leq n$, we wish to extend to a coherent $\mathcal{F}_k$-family of types. If there is any set $S\subseteq [n]$ with $|S| = k$ such that $S\notin \mathcal{F}_{k-1}$, then all proper subsets of $S$ are in $\mathcal{F}_{k-1}$. Hence we have types $\{p_R \mid R\in \mathcal{P}^-(S)\}$ which form a coherent $\mathcal{P}^-(S)$-family. Using $k$-amalgamation, we can find a non-redundant type $p_S$ in the variables $\overline{x}_S$ extending the types $p_R$. Doing this for all such $S$ gives a coherent $\mathcal{F}_k$-family of types, as desired.
\end{proof}

Disjoint $n$-amalgamation is more general and seems more natural, but it is basic disjoint $n$-amalgamation which is relevant in the proof of Theorem~\ref{thm:main}. We are largely interested in theories with disjoint $n$-amalgamation for all $n$, and in this case the two notions agree.

\begin{proposition}\label{prop:equivalence}
$T$ has disjoint $n$-amalgamation for all $n$ if and only if $T$ has basic disjoint $n$-amalgamation for all $n$.
\end{proposition}
\begin{proof}
One direction is clear, since basic disjoint $n$-amalgamation is a special case of disjoint $n$-amalgamation.

In the other direction, note first that there is a solution to the disjoint $n$-amalgamation problem $\{p_S\mid S\in\mathcal{P}^-([n])\}$ if and only if the partial type $$\{x\neq x'\mid x,x'\text{ distinct}\}\cup \bigcup_{S\in\mathcal{P}^-([n])}p_S(\overline{x}_S)$$ is consistent (actually, we could omit the formulas asserting non-redundancy when $n > 2$). Hence, by compactness, we can reduce to the case that $A$ is finite and each tuple of variables $\overline{x}_i$ is finite.

Let $N = |A| + \sum_{i = 1}^n |\overline{x}_i|$, where $|\overline{x}_i|$ is the length of the tuple $\overline{x}_i$. Introduce variables $y_1,\dots, y_N$, where $y_1,\dots, y_{|A|}$ enumerate $A$ and the remaining variables relabel the $x$ variables. Now each type $p_S$ over $A$ determines a type in some subset of the $y$ variables, by replacing the parameters from $A$ and the $x$ variables by the appropriate $y$ variables. Closing downward under restriction to smaller sets of variables, we obtain a partial basic disjoint $N$-amalgamation problem over the empty set in the singleton variables $y_1,\dots, y_N$. By Lemma~\ref{lem:partial} and basic disjoint $N$-amalgamation, this partial amalgamation problem has a solution, a type $p_{[N]}(y_1,\dots,y_N)$ over the empty set. Once again replacing the $y$ variables with the original parameters from $A$ and $x$ variables, we obtain a type $p_{[n]}$ over $A$ which is a solution to the original $n$-amalgamation problem.
\end{proof}

\begin{example}
The class $\mathcal{G}_\triangle$ of triangle-free graphs has disjoint $2$-amalgamation: if $A$ embeds in $B$ and $C$, we can amalgamate $B$ and $C$ ``freely" over $A$ by not adding any new edge relations between $B$ and $C$. But it does not have disjoint $3$-amalgamation: the non-redundant $2$-types determined by $x_1Rx_2$, $x_2Rx_3$, and $x_1Rx_3$ cannot be amalgamated. 

Generalizing, let $K_n^k$ be the class of $n$-free $k$-hypergraphs: the language consists of a single $k$-ary relation $R(x_1,\dots,x_k)$, and the structures in $K_n^k$ are hypergraphs (so $R$ is symmetric and anti-reflexive) such that for every $n$-tuple $\overline{a}$ of distinct elements, there is some subtuple $\overline{b}$ of length $k$ such that $\lnot R(\overline{b})$ holds. Note that $\mathcal{G}_\triangle$ is $K_3^2$.

For $n>k$, $K_n^k$ satisfies basic disjoint $m$-amalgamation for $m<n$, but fails basic disjoint $n$-amalgamation, since the first forbidden configuration has size $n$. However, $K_n^k$ already fails disjoint $(k+1)$-amalgamation. Over a base set $A$ consisting of a complete hypergraph on $(n-k-1)$ vertices, the $k$-type over $A$ which describes, together with $A$, a complete hypergraph on $(n-1)$ vertices is consistent, but $(k+1)$ copies of it cannot be amalgamated. 
\end{example}

\begin{example}\label{ex:reducts}
There are countably categorical theories which do not have disjoint $n$-amalgamation for all $n$, but which admit countably categorical expansions with disjoint $n$-amalgamation for all $n$.

As a simple example, consider the theory of a single equivalence relation with $k$ infinite classes. Transitivity is a failure of disjoint $3$-amalgamation: the non-redundant $2$-types determined by $x_1Ex_2$, $x_2Ex_3$, and $\lnot x_1Ex_3$ cannot be amalgamated. But if we expand the language by adding $k$ new unary relations $C_1,\dots,C_k$ in such a way that each class is named by one of the $C_i$, the resulting theory has disjoint $n$-amalgamation for all $n$.

For a more interesting example, the random graph (which is easily seen to have disjoint $n$-amalgamation for all $n$) in its canonical language has a reduct to a ultrahomogeneous $3$-hypergraph, where the relation $R(a,b,c)$ holds if and only if there are an \emph{odd} number of the three possible edges between $a$, $b$, and $c$. This structure turns out to be ultrahomogeneous in the language $\{R\}$, and its age is the class of all finite $3$-hypergraphs with the property that on any four vertices $a$, $b$, $c$, and $d$, there are an \emph{even} number of the four possible $3$-edges. Hence this class fails to have disjoint $4$-amalgamation. For more information on this example, see \cite{MSurvey}, where it is called the homogeneous two-graph. More examples of this kind can be found in the literature on reducts of homogeneous structures, e.g.\ \cite{Thomas}.
\end{example}

\subsection{Pseudofiniteness}\label{sec:pseudofinite}

\begin{definition}
A theory $T$ is \emph{pseudofinite} if for every sentence $\varphi$ such that $T\models \varphi$, $\varphi$ has a finite model.
\end{definition}

Theorem~\ref{thm:main} below is stated in a fine-grained way: amalgamation just up to level $n$ gives pseudofiniteness of the theory $T_{K,n}$ (see Definition~\ref{def:TKn}). The proof involves a probabilistic construction of a structure of size $N$ for each $N$ ``from the bottom up". This is the same idea as in the proof of Lemma~\ref{lem:partial}, but there we could fix an arbitrary $k$-type extending a given coherent family of $l$-types for $l<k$. Here we introduce randomness by choosing an extension uniformly at random. 

The probabilistic calculation is essentially the same as the one used in the classical proofs of the zero-one laws for graphs and general $L$-structures (see \cite[Lemma 7.4.6]{Hodges}). The key point is that the amalgamation properties allow us to make all choices as independently as possible: the quantifier-free types assigned to subsets $A$ and $B$ of $[N]$ are independent when conditioned on the quantifier-free type assigned to $A\cap B$. It is this independence which makes the calculation go through.

Formally, we construct a probability measure on the space $L[N]$ of $L$-structures with domain $[N]$. Given a formula $\varphi(\overline{x})$ and a tuple $\overline{a}$ from $[N]$, we write $[\varphi(\overline{a})] = \{M\in L[N]\mid M\models \varphi(\overline{a})\}$. The space $L[N]$ is topologized by taking the instances of the atomic and negated atomic formulas $[(\lnot) R(\overline{a})]$ as subbasic open sets. Of course, if $L$ is finite, then $L[N]$ is a finite discrete space.

\begin{theorem}\label{thm:main}
Let $K$ be a \Fraisse class whose generic theory $T_K$ has basic disjoint $k$-amalgamation for all $2\leq k\leq n$. Then every sentence in $T_{K,n}$ has a finite model. If $T_K$ has basic disjoint $k$-amalgamation for all $k$, then every sentence in $T_K$ has a finite model in $K$.
\end{theorem}
\begin{proof}
I will define a probability measure $\mu_N$ on $L[N]$ for each $N\in\omega$ by describing a probabilistic construction of a structure $M_N\in L[N]$. Recall the notation above and in Remark~\ref{rem:fraissenotation}.

We assign quantifier-free $k$-types to each subset of size $k$ from $[N]$ by induction. When $k = 0$, there is no choice: by hereditarity and the joint embedding property, there is a unique empty structure in $K(0)$. When $k=1$, for each $i\in [N]$, choose the quantifier-free $1$-type of $\{i\}$ uniformly at random from $K(1)$. Now proceed inductively: having assigned quantifier-free $l$-types to all subsets of size $l$ with $l<k$, we wish to assign quantifier-free $k$-types. For each $k$-tuple $i_1,\dots,i_k$ of distinct elements from $[N]$, let $P = \{p_S\mid S\in \mathcal{P}^-([k])\}$ be the collection of quantifier-free types assigned to all proper subtuples, i.e.\ $p_S(\overline{x}_S) = \mathrm{qftp}(\{i_j\mid j\in S\})$. If $T_K$ has basic disjoint $k$-amalgamation, $K(k,P)$ is nonempty and finite, and we may choose the quantifier-free $k$-type of $i_1,\dots,i_k$ uniformly at random from $K(k,P)$.

Now if $T_K$ has basic disjoint $k$-amalgamation for all $k$, we can continue this construction all the way up to $k = N$, so that the resulting structure $M_N$ is in $K(N)$. Call this the unbounded case. On the other hand, if $T_K$ has basic disjoint $k$-amalgamation only for $k\leq n$, then we stop at $k = n$. To complete the construction, we assign any remaining relations completely freely at random. That is, for each relation $R$ (of arity $r > n$) and $r$-tuple $i_1,\dots,i_r$ containing at least $n+1$ distinct elements, we set $R(i_1,\dots,i_r)$ with probability $1/2$. The result is an $L$-structure $M_N$ which may not be in $K$, but the induced structures of size at most $n$ are guaranteed to be in $K$. Call this the bounded case.

I claim that if $\varphi$ is one of the axioms of $T_{K,n}$ (in the bounded case) or $T_K$ (in the unbounded case), then $\lim_{N\rightarrow \infty} \mu_N([\varphi]) = 1$. 

Each universal axiom $\varphi$ has the form $\A{x_1,\dots,x_k} \psi(\overline{x})$ (with $k\leq n$ in the bounded case), where $\psi$ is quantifier-free and true on all $k$-tuples from structures in $K$. Since all substructures of our random structure of size at most $k$ are in $K$, $\varphi$ is always satisfied by $M_N$, and so $\mu_N([\varphi]) = 1$ for all $N$.

Now suppose $\varphi$ is the one-point extension axiom $\A{\overline{x}} \E{y} (\theta_A(\overline{x}) \rightarrow \theta_B(\overline{x},y)).$ Let $\overline{a}$ be a tuple of $|A|$-many distinct elements from $[N]$ and $b$ any other element. Conditioning on the event that $M_n\models \theta_A(\overline{a})$, there is a positive probability $\varepsilon$ that $M_N\models \theta_B(\overline{a},b)$.

Indeed, in the unbounded case, or when $|A|< n$ in the bounded case, $\theta_B$ specifies the quantifier-free $|B|$-type of the tuple $\overline{a}b$ among those allowed by $K$. There is a positive probability ($1/|K(1)|$) that the correct $1$-type is assigned to $b$, and, given that the correct $l$-type has been assigned to all subtuples of $\overline{a}b$ involving $b$ of length $l<k$, there is a positive probability ($1/|K(k,P)|$ for the appropriate basic disjoint $k$-amalgamation problem $P$) that the correct $k$-type is assigned to a given subtuple of length $k$. Then $\varepsilon$ is the product of all these probabilities for $1\leq k \leq |B|$. When $|A|\geq n$ in the unbounded case, the above reasoning applies for the subtuples of $\overline{a}b$ of length at most $n$. On longer tuples, since $\theta_B$ only mentions finitely many relations, and the truth values of these relations are assigned freely at random, there is some additional positive probability that these will be decided in a way satisfying $\theta_B$ (at least $1/2^m$, where $m$ is the minimum number of additional instances of relations which need to be decided positively or negatively to ensure satisfaction of $\theta_B$). 

Moreover, for distinct elements $b$ and $b'$, the events that $\overline{a}b$ and $\overline{a}b'$ satisfy $\theta_B$ are conditionally independent, since the quantifier-free types of tuples involving elements from $\overline{a}$ and $b$ but not $b'$ are decided independently from those of tuples involving elements from $\overline{a}$ and $b'$ but not $b$, conditioned on the quantifier-free type assigned to $\overline{a}$.

Now we compute the probability that $\varphi$ is \emph{not} satisfied by $M_N$. Conditioned on the event that $M_N\models \theta_A(\overline{a})$, the probability that $M_N\not\models \E{y} \theta_B(\overline{a},y)$ is $(1-\varepsilon)^{N-|A|}$, since there are $N-|A|$ choices for the element $b$, each with independent probability $(1-\varepsilon)$ of failing to satisfy $\theta_B$. Removing the conditioning, the probability that $M_N\not\models \E{y}(\theta_A(\overline{a})\rightarrow \theta_B(\overline{a},y))$ for any given $\overline{a}$ is at most $(1-\varepsilon)^{N-|A|}$, since the formula is vacuously satisfied when $\overline{a}$ does not satisfy $\theta_A$. Finally, there are $N^{|A|}$ possible tuples $\overline{a}$, so the probability that $M_N\not\models \A{\overline{x}}\E{y} (\theta_A(\overline{x})\rightarrow \theta_B(\overline{x},y))$ is at most $N^{|A|}(1-\varepsilon)^{N-|A|}$. Since $|A|$ is constant, the exponential decay dominates the polynomial growth, and $\lim_{N\to\infty} \mu_N([\lnot \varphi]) = 0$, so $\lim_{N\to\infty} \mu_N([\varphi]) = 1$. 

To conclude, any sentence $\psi\in T_{K,n}$ is a logical consequence of finitely many of the axioms $\varphi_1,\dots,\varphi_m$ considered above. We need only pick $N$ large enough so that $\mu_N([\varphi_i]) > 1-1/m$ for all $i$. Then $\mu_N([\bigwedge_{i = 1}^m\varphi_i]) >0$, so the conjunction $\bigwedge_{i = 1}^m\varphi_i$, and hence also $\psi$, has a model of size $N$. In the unbounded case, our construction ensures that this model is in $K$.
\end{proof}

\begin{corollary}\label{cor:main}
Any countably categorical theory $T$ with disjoint $n$-amalgamation for all $n\geq 2$ is pseudofinite.
\end{corollary}
\begin{proof}
Let $T'$ be the equivalent of $T$ in the canonical language. Then it suffices to show that $T'$ is pseudofinite, since pseudofiniteness is preserved under interdefinability. But $T'$ is the generic theory for a \Fraisse class with basic disjoint $n$ amalgamation for all $n$, so by Theorem~\ref{thm:main}, it is pseudofinite.
\end{proof}

\begin{remark}\label{rem:reduct}
Since pseudofiniteness is preserved under reduct, the examples described in Example~\ref{ex:reducts} are pseudofinite. 
\end{remark}

\subsection{Relationship to other notions}\label{sec:notions}

The notion of $n$-amalgamation has been studied in other model-theoretic contexts, usually in the form of \emph{independent $n$-amalgamation}. Given some notion of independence, $\ind$, the main example being nonforking independence in a simple theory, an independent $n$-amalgamation problem is given by a coherent $\mathcal{P}^-([n])$-family of types over $A$, with the non-redundancy condition replaced by the condition that any realization $\{\overline{a}_i\mid i\in S\}$ of $p_S(\overline{x}_S)$ is an independent set over $A$ with respect to $\ind$.

In the case $n = 3$, independent $3$-amalgamation over models is often called the independence theorem. It is a well-known theorem of Kim and Pillay that the independence theorem, along with a few other natural properties, characterizes forking in simple theories.

\begin{theorem}[\cite{KP2}, Theorem 4.2]\label{thm:KP}
Let $T$ be a complete theory and $\ind$ a ternary relation, written $a\ind_A B$, where $a$ is a finite tuple and $A$ and $B$ are sets. As usual, all tuples and sets come from some highly saturated model of $T$. Suppose that $\ind$ satisfies the following properties:
\begin{enumerate}
\item (Invariance) If $a\ind_A B$ and $\mathrm{tp}(a'A'B') = \mathrm{tp}(aAB)$, then $a'\ind_{A'} B'$.
\item (Local character) For all $a,B$, there is $A\subseteq B$ such that $|A| \leq |T|$ and $a\ind_{A} B$.
\item (Finite character) $a\ind_A B$ if and only if for every finite tuple $b$ from $B$, $a\ind_A Ab$.
\item (Extension) For all $a$, $A$, and $B$, there is $a'$ such that $\mathrm{tp}(a'/A) = \mathrm{tp}(a/A)$ and $a'\ind_A B$.
\item (Symmetry) If $a\ind_A Ab$, then $b\ind_A Aa$.
\item (Transitivity) If $A\subseteq B\subseteq C$, then $a\ind_A B$ and $a\ind_{B} C$ if and only if $a\ind_A C$.
\item (Independence theorem) Let $M\models T$ be a model, $a$ and $a'$ tuples such that $\mathrm{tp}(a/M) = \mathrm{tp}(a'/M)$ and $A$ and $B$ sets. If $A\ind_M B$, $a\ind_M A$, and $a'\ind_M B$, then there exists $a''$ such that $\mathrm{tp}(Aa''/M) = \mathrm{tp}(Aa/M)$, $\mathrm{tp}(Ba''/M) = \mathrm{tp}(Ba'/M)$, and $a''\ind_M AB$. 
\end{enumerate}
Then $T$ is simple, and $\ind$ is nonforking ($\ind = \ind[f]$).
\end{theorem}

Disjoint $n$-amalgamation is a strong form of independent amalgamation, where the relevant independence relation is the disjointness relation $\ind[=]$, defined by $A\ind[=]_C B$ if and only if $A\cap B \subseteq C$. We say a theory has trivial forking if $\ind[f] = \ind[=]$. 

\begin{theorem}\label{thm:simple}
A countably categorical theory $T$ with disjoint $2$-amalgamation (\ie trivial acl) and disjoint $3$-amalgamation is simple with trivial forking.
\end{theorem}
\begin{proof}
We can use Theorem~\ref{thm:KP} to show that $\ind[f] = \ind[=]$. Most of the conditions are straightforward to check, so I'll only remark on a few of them. For local character, we can take $A = a\cap B$, so $A$ is finite and $a\ind[=]_A B$. For extension, we find $a'$ by realizing the type $\mathrm{tp}(a/A)\cup \{a_i\neq b\mid a_i\text{ from }a \text{ such that }a_i\notin A, \text{ and } b\in B\}$. This is consistent by trivial acl and compactness. Finally, for the independence theorem, we apply disjoint $3$-amalgamation to amalgamate the three $2$-types $p_{\{12\}} = \mathrm{tp}(aA/M)$, $p_{\{13\}} = \mathrm{tp}(a'B/M)$, $p_{\{23\}} = \mathrm{tp}(AB/M)$ (first removing any redundant elements of $M$ from $a$, $a'$, $A$, and $B$).
\end{proof}

\begin{remark}\label{rem:no3exp}
A consequence of Theorem~\ref{thm:simple} is that the class of triangle-free graphs $\mathcal{G}_\triangle$ does not admit a \Fraisse expansion to a class with disjoint $n$-amalgamation for all $n$, since its generic theory is not simple \cite{Shelah500}.
\end{remark}

\begin{remark}\label{rem:nsimple}
Motivated by the fact that many examples of simple theories (such as $\TRG$ and ACFA \cite{CH}) satisfy independent $n$-amalgamation for $n \geq 3$, Kolesnikov~\cite{Kolesnikov} and Kim, Kolesnikov, and Tsuboi~\cite{KKT} developed a hierarchy of notions of $n$-simplicity for $1\leq n \leq \omega$, where $1$-simplicity coincides with simplicity. If a countably categorical theory $T$ has disjoint $k$-amalgamation for all $2\leq k\leq n$, then it is $(n-2)$-simple with trivial forking, and if it has disjoint $n$-amalgamation for all $n$, then it is $\omega$-simple.
\end{remark}

Several other appearances of $n$-amalgamation properties in model theory are worth mentioning. In the context of abstract elementary classes, independent $n$-amalgamation of models goes by the name ``excellence" (see \cite{BaldwinBook}, for example). Disjoint $n$-amalgamation for classes of finite structures has also been studied by Baldwin, Koerwien, and Laskowski with applications to AECs~\cite{BKL}. And in the context of stable theories, Goodrick, Kim, and Kolesnikov have uncovered a connection between existence and uniqueness of independent $n$-amalgamation and definable polygroupoids \cite{GKK}, generalizing earlier work of Hrushovski on independent $3$-amalgamation and groupoids \cite{HGroupoids}.

The observation that disjoint $n$-amalgamation is sufficient for pseudofiniteness generalizes and unifies a number of earlier observations. I will note a few here:
\begin{itemize}
\item Oberschelp \cite{Oberschelp} identified an unusual syntactic condition which is sufficient for the almost-sure theory of a class of finite structures under the uniform measures to agree with its generic theory. A universal sentence is called \emph{parametric} if it is of the form $\A{x_1,\dots,x_n} ((\bigwedge_{i\neq j} x_i\neq x_j) \rightarrow \varphi(\overline{x}))$ where $\varphi$ is a Boolean combination of atomic formulas $R(y_1,\dots,y_m)$ such that each variable $x_i$ appears among the $y_j$. For example, reflexivity $\A{x} R(x,x)$ and symmetry $\A{x,y} (x\neq y \rightarrow (R(x,y) \leftrightarrow R(y,x)))$ are parametric conditions, while transitivity $\A{x,y,z} ((R(x,y)\land R(y,z))\rightarrow R(x,z))$ is not a parametric condition, since each atomic formula appearing only involves two of the three quantified variables. A \emph{parametric class} is the class of finite models of a set of parametric axioms. 

Any parametric class has disjoint $n$-amalgamation for all $n$. It is easiest to see this by checking basic disjoint $n$-amalgamation: the restrictions imposed by a parametric theory on the relations involving non-redundant $n$-tuples and $m$-tuples are totally independent when $n\neq m$.

\item In their work on the random simplicial complex, Brooke-Taylor \& Testa \cite{BTT} introduced the notion of a \emph{local \Fraisse class} and showed that the generic theory of a local \Fraisse class is pseudofinite, by methods similar to those in the proof of Theorem~\ref{thm:main}. A universal sentence is called \emph{local} if it is of the form $\A{x_1,\dots,x_n} (R(x_1,\dots,x_n)\rightarrow \psi(\overline{x}))$, where $R$ is a relation in the language and $\psi$ is quantifier-free. A \emph{local class} is the class of finite models of a set of local axioms.

Again, any local class has $n$-amalgamation for all $n$. A local theory only imposes restrictions on tuples which satisfy some relation. So disjoint $n$-amalgamation problems can be solved ``freely" by simply not adding any further relations.

\item In \cite{Ahlman}, Ahlman has investigated countably categorical theories in a binary relational language (one with no relation symbols of arity greater than $2$) which are simple with $\mathrm{SU}$-rank $1$ and trivial pregeometry. In the case when $\mathrm{acl}^{\mathrm{eq}}(\varnothing) = \varnothing$, this agrees with what I call a simple theory with trivial forking ($\ind[f] = \ind[=]$) above.

Ahlman shows that in such a theory $T$ there is a $\varnothing$-definable equivalence relation $\xi$ with finitely many infinite classes such that $T$ can be axiomatized by certain ``$(\xi,\Delta)$-extension properties" describing the possible relationships between elements in different classes. Further, he shows that these theories are pseudofinite. The definition of $(\xi,\Delta)$-extension property is somewhat technical, so I will not give it here. But this condition implies that $T$ has an expansion (obtained by naming the finitely many classes of $\xi$) with $n$-amalgamation for all $n$. The fact that the language is binary ensures that describing the possible relationships between pairs of elements suffices.
\end{itemize}

\section{Generic theories of equivalence relations}\label{sec:eqrel}

\subsection{Filtered \Fraisse classes}\label{sec:filtered}

We will extend the disjoint $n$-amalgamation argument for pseudofiniteness to certain non-simple theories, using the notion of a filtered \Fraisse class.

\begin{definition}\label{def:filtered}
A \Fraisse class K is \emph{filtered} by a chain $K_0\subseteq K_1\subseteq K_2\subseteq \dots $ if each $K_n$ is a \Fraisse class, and $\bigcup_{n\in\omega} K_n = K$.
\end{definition}

\begin{theorem} \label{thm:filtered}
Let $K$ be a \Fraisse class filtered by $\{K_n\mid n\in\omega\}$. Then $\varphi\in T_K$ if and only if $\varphi\in T_{K_n}$ for all sufficiently large $n$. 
\end{theorem}
\begin{proof}
It suffices to check for each of the axioms of $T_K$ given in Theorem~\ref{thm:axiomatization}. Since each $K_n$ is a subclass of $K$, every universal sentence in $T_K$ is also in $T_{K_n}$. Let $(A,B)$ be a one-point extension with corresponding axiom $\varphi$. For large enough $n$, the structures $A$ and $B$ are in $K_n$, so $(A,B)$ is also a one-point extension in $K_n$, and $\varphi\in T_{K_n}$.
\end{proof}

Pseudofiniteness is preserved in filtered \Fraisse classes.

\begin{corollary}\label{cor:filtered}
If a \Fraisse class $K$ is filtered by $\{K_n\mid n\in\omega\}$ and each generic theory $T_{K_n}$ is pseudofinite, then the generic theory $T_K$ is pseudofinite.
\end{corollary}
\begin{proof}
Each sentence $\varphi$ in $T_K$ is also in $T_{K_n}$ for sufficiently large $n$, and hence $\varphi$ has a finite model.
\end{proof}

As a consequence, if $K$ is filtered by $\{K_n\mid n\in\omega\}$, and each $K_n$ admits a \Fraisse expansion with disjoint $n$-amalgamation for all $n$, then $K$ is pseudofinite. This argument is used in the next two sections to establish pseudofiniteness of the theories $\Tfeq$ and $\TCPZ$. 

It is worth noting that this method cannot be used to show that the theory of the generic triangle-free graph is pseudofinite. Let $G_1$, $G_2$, and $G_3$ be the graphs on three vertices with a single edge, two edges, and three edges respectively. For any filtration $\{K_n\mid n\in\omega\}$ of the \Fraisse class $\mathcal{G}_\triangle$ of triangle-free graphs, some $K_n$ must include the graphs $G_1$ and $G_2$ but not $G_3$. But Proposition~\ref{prop:nohenson} shows that such a class does not admit a \Fraisse expansion with disjoint $n$-amalgamation for all $n$.

\begin{proposition}\label{prop:nohenson}
Let $K$ be a \Fraisse class consisting of graphs (in the language with a single edge relation $E$), and suppose that $K$ contains the graphs $G_1$ and $G_2$ but not $G_3$. Then no \Fraisse expansion of $K$ has disjoint $2$-amalgamation and disjoint $3$-amalgamation. 
\end{proposition}
\begin{proof}
Suppose for contradiction that $K$ has a \Fraisse expansion $K'$ in the language $L'$ with disjoint $2$-amalgamation and disjoint $3$-amalgamation. Let $p(x)$ be any quantifier-free $1$-type in $K'$. Then by disjoint $2$-amalgamation we can find some quantifier-free $2$-type $q(x,y)$ in $K'$ such that $q(x,y)\models p(x)\land p(y)\land x\neq y$. Now, letting $p_\varnothing$ be the unique quantifier-free $0$-type in $K'$, the family of types $\{p_\varnothing, p(x),p(y),p(z),q(x,y),q(y,z),q(x,z)\}$ is a basic disjoint $3$-amalgamation problem for $K'$. Then we must have $q(x,y)\models \lnot xRy$, for otherwise the reduct to $L$ of any solution to the $3$-amalgamation problem would be a copy of $G_3$ in $K$.

Let $H$ be the graph on two vertices, $v_1$ and $v_2$, with no edge. Note that $H$ is in $K$. Labeling the vertices of $G_1$ by $v_1$, $v_2$, $v_3$, so that the unique edge is $v_2Rv_3$, $G_1$ is a one-point extension of $H$. Now $H$ admits an expansion to a structure in $K'$ (described by $q(v_1,v_2)$) in which both vertices $v_1$ and $v_2$ have quantifier-free type $p$, so since $K'$ is a \Fraisse expansion of $K$, $G_1$ admits a compatible expansion to a structure in $K'$, call it $G_1'$. Let $p'(y) = \mathrm{qftp}_{G_1'}(v_3)$, and let $q_i(x,y) = \mathrm{qftp}_{G_1'}(v_i,v_3)$ for $i = 1,2$. Note that we have $q_i(x,y)\models p(x)\land p'(y)$ for $i = 1,2$, but $q_1(x,y)\models \lnot xRy$, while $q_2(x,y)\models xRy$. That is, the pair of $1$-types $p(x)$ and $p'(y)$ are consistent with both $xRy$ and $\lnot xRy$. We will use this situation to build a triangle.

Again, labeling the vertices of $G_2$ by $v_1$, $v_2$, $v_3$, so that $\lnot v_1Rv_2$, $G_2$ is a one-point extension of $H$. Since $H$ admits an expansion to a structure $H'$ in $K'$ so $H'\models q_1(v_1,v_2)$, $G_2$ admits a compatible expansion to a structure $G_2'$ in $K'$. Let $p'' = \mathrm{qftp}_{G_2'}(v_3)$, $r_1(x,z) = \mathrm{qftp}_{G_2'}(v_1,v_3)$, and $r_2(y,z) = \mathrm{qftp}_{G_2'}(v_2,v_3)$. Note that $r_1(x,z)\models p(x)\land p''(z)\land xRz$ and $r_2(y,z)\models p'(x) \land p''(z)\land yRz$.

Now the family of types $\{p_\varnothing, p(x),p'(y),p''(z),q_2(x,y),r_1(x,z),r_2(y,z)\}$ is a basic disjoint $3$-amalgamation problem for $K'$. But the reduct to $L$ of any solution is a copy of $G_3$ in $K$.
\end{proof}

\subsection{The theory \texorpdfstring{$\Tfeq$}{T*\_feq}}\label{sec:Tfeq}

Let $L$ be the language with two sorts, $O$ and $P$ (for ``objects" and ``parameters"), and a ternary relation $E_x(y,z)$, where $x$ is a variable of sort $P$ and $y$ and $z$ are variables of sort $O$. Then $\Kfeq$ is the class of finite $L$-structures with the property that for all $a$ of sort $P$, $E_a(y,z)$ is an equivalence relation on $O$.

$\Kfeq$ is a \Fraisse class. We define $\Tfeq$ to be the generic theory of $\Kfeq$. Our aim is to show that it is pseudofinite. Before giving the details of the proof, I will describe the simple idea: filter the class $\Kfeq$ by the subclasses $K_n$ in which each equivalence relation in the parameterized family has at most $n$ classes. Expand these classes by parameterized predicates naming each class. The resulting class has $n$-amalgamation for all $n$, and hence has pseudofinite generic theory. 

\begin{theorem}\label{thm:tfeq}
$\Tfeq$ is pseudofinite.
\end{theorem}
\begin{proof}
For $n\geq 1$, let $K_n$ be the subclass of $\Kfeq$ consisting of those structures with the property that for all $a$ of sort $P$, the equivalence relation $E_a$ has at most $n$ classes. Let's check that $K_n$ is a \Fraisse class. 

It clearly has the hereditary property. For the disjoint amalgamation property, suppose we have embeddings $f\colon A\hookrightarrow B$ and $g\colon A\hookrightarrow C$ of structures in $K_n$. We specify a structure $D$ with domain $A\cup (B\setminus f[A])\cup (C\setminus g[A])$ into which $B$ and $C$ embed in the obvious way over $A$. That is, for each parameter $a$ in $P(D)$, we must specify an equivalence relation on $O(D)$. If $a$ is in $P(A)$, it already defines equivalence relations on $B$ and $C$. First, number the $E_a$-classes in $A$ by $1,\dots,l$. Then, if there are further unnumbered $E_a$-classes in $B$ and $C$, number them by $l+1,\dots,m_B$ and $l+1,\dots,m_C$ respectively. Note that $m_B,m_C\leq n$. Now define $E_a$ in $O(D)$ to have $\mathrm{max}(m_B,m_C)$ classes by merging the classes assigned the same number in the obvious way. The situation is even simpler if $a$ is not in $P(A)$. Say without loss of generality it is in $P(B)$. Then we can extend $E_a$ to $O(C)$ by adding all elements of $O(C\setminus g[A])$ to a single existing $E_a$-class. The joint embedding property follows from the amalgamation property by taking $A$ to be the empty structure.

For any structure $A$ in $\Kfeq$, if $|O(A)| = N$, then for all $a\in P(A)$, the equivalence relation $E_a$ has at most $N$ classes, so $A\in K_N$. Hence $\Kfeq = \bigcup_{n = 1}^\infty K_n$. So $\Kfeq$ is a filtered \Fraisse class, and by Corollary~\ref{cor:filtered}, it suffices to show that each $T_{K_n}$ is pseudofinite.

Let $L_n'$ be the expanded language which includes, in addition to the relation $E$, $n$ binary relation symbols $C_1(x,y), \dots, C_n(x,y)$, where $x$ is a variable of sort $P$ and $y$ is a variable of sort $O$. Let $K_n'$ be the class of finite $L_n'$-structures which are expansions of structures in $K_n$ such that for all $a$ of sort $P$, each of the $E_a$-classes is picked by one of the formulas $C_i(a,y)$. 

We need to check that $K_n'$ is a \Fraisse expansion of $K_n$. Certainly we have $K_n = \{A\restriction L \mid A\in K_n'\}$, since every structure in $K_n$ can be expanded to one in $K_n'$ by labeling the classes for each equivalence relation arbitrarily. Suppose now that $(A,B)$ is a one-point extension in $K_n$ and $A'$ is an expansion of $A$ to a structure in $K_n'$. If the new element $b\in B$ is in $P(B)$, then it defines a new equivalence relation $E_b$ on $O(A) = O(B)$, and we can expand $B$ to $B'$ in $K_n'$ by labeling the $E_b$-classes arbitrarily. On the other hand, suppose $b$ is in $O(B)$. Then for each parameter $a$, either $b$ is an existing $E_a$-class labeled by $C_i(a,y)$, in which case we set $C_i(a,b)$, or $b$ is in a new $E_a$-class, in which case we set $C_j(a,b)$ for some unused $C_j$.

Finally, note that $T_{K_n'}$ has disjoint $2$-amalgamation, since it is a \Fraisse class with the disjoint amalgamation property. I claim that it also has disjoint $n$-amalgamation for all $n\geq 3$. Indeed, the behavior of the ternary relation $E_x(y,z)$ is entirely determined by the behavior of the binary relations $C_i(x,y)$, and an $L_{n}'$-structure $(P(A),O(A))$ is in $K_n'$ if and only if for every $a$ in $P(A)$ and $b$ in $O(a)$, $C_i(a,b)$ holds for exactly one $i$. So any inconsistency is already ruled out at the level of the $2$-types. Since in a coherent $\mathcal{P}^-([n])$-family of types for $n\geq 3$, every pair of variables is contained in one of the types, we conclude that there are no inconsistencies, and every disjoint $n$-amalgamation problem has a solution.

So $T_{K_n'}$ has disjoint $n$-amalgamation for all $n$, and hence it and its reduct $T_{K_n}$ are pseudofinite by Theorem~\ref{thm:main}.
\end{proof}

A natural question is whether $\Tfeq$ is, in fact, the almost-sure theory for the class $\Kfeq$ for the uniform measures. It is not, as the following proposition shows. Of course, since we have described $\Kfeq$ in a two-sorted language, there is some ambiguity as to what we mean by the uniform measures. For maximum generality, let us fix two increasing functions $f,g\colon \omega\to \omega$. For $n\in \omega$, let $\Kfeq(f(n),g(n))$ be the structures in $\Kfeq$ with object sort of size $f(n)$ and parameter sort of size $g(n)$, and let $\mu_{f(n),g(n)}$ be the uniform measure on $\Kfeq(f(n),g(n))$.

\begin{proposition}\label{prop:notuniform1}
There is a sentence $\varphi$ in $\Tfeq$ such that $$\lim_{n\to \infty} \mu_{f(n),g(n)}(\{A\in\Kfeq(f(n),g(n))\mid A\models\varphi\}) = 0.$$
\end{proposition}
\begin{proof}
An example of such a sentence $\varphi$ is $$\A{(x:P)}\A{(x':P)}\A{(y:O)}\A{(y':O)}\E{(z:O)} (y\neq y'\rightarrow (E_x(y,z)\land E_{x'}(y',z))),$$ which expresses that any two equivalence classes for distinct equivalence relations intersect. $\varphi$ is in $\Tfeq$, since for any structure in $\Kfeq$ with parameters $a\neq a'$ and objects $b,b'$ (possibly $b = b'$), we can add a new object element $c$ which is $E_a$-equivalent to $b$ and $E_{a'}$-equivalent to $b'$, so $\varphi$ is implied by the relevant one-point extension axioms.

I will sketch the asymptotics: the measure $\mu_{f(n),g(n)}$ on amounts to picking $g(n)$ equivalence relations on a set of size $f(n)$ uniformly and independently. The expected number of equivalence classes in an equivalence relation on a set of size $n$, chosen uniformly, grows asymptotically as $\frac{n}{\log(n)}(1 + o(1))$ \cite[Proposition VIII.8]{FS}. Thus, most of the $g(n)$ equivalence relations have equivalence classes which are much smaller (with average size approximately $\log(n)$) than the number of classes, and the probability that every $E_a$-class is large enough to intersect every $E_b$-class nontrivially for all distinct $a$ and $b$ converges to $0$.
\end{proof}

Proposition~\ref{prop:notuniform1} shows that $\Tfeq$ is not the almost-sure theory of $\Kfeq$ for the measures $\mu_{f(n),g(n)}$, but it would be interesting to know whether such an almost-sure theory exists.

\begin{question}\label{question:feq01}
Does the class $\Kfeq$ have a first-order zero-one law for the measures $\mu_{f(n),g(n)}$? If so, does the almost-sure theory depend on the relative growth rates of $f$ and $g$?
\end{question}

\subsection{The theory \texorpdfstring{$\TCPZ$}{T\_CPZ}}\label{sec:TCPZ}

Let $L$ be the language with a symbol $E_n(\overline{x};\overline{y})$ of arity $2n$ for all $n\geq 1$. Then $\KCPZ$ is the class of finite $L$-structures with the property that $E_n$ is an equivalence relation on $n$-tuples for all $n$, and there is a single $E_n$-class consisting of all $n$-tuples which do \emph{not} consist of $n$ distinct elements. 

$\KCPZ$ is a \Fraisse class. We define $\TCPZ$ to be the generic theory of $\KCPZ$. In \cite{CPZ}, Casanovas, Pel\'{a}ez, and Ziegler introduced the theory $\TCPZ$ and showed that it has NSOP$_2$ and is not simple. For completeness, I will show how to combine the ``independence lemma" from \cite{CPZ} with the $3$-amalgamation criterion due to Chernikov and Ramsey \cite{CR} to show that, in fact, $\TCPZ$ has NSOP$_1$. 

We write $\ind[u]$ for coheir independence: given a model $M$ and tuples $a$ and $b$, $a\ind[u]_M b$ if and only if $\mathrm{tp}(a/Mb)$ is finitely satisfiable in $M$; that is, for every formula $\varphi(x,m,b)\in \mathrm{tp}(a/Mb)$, there exists $m'\in M$ such that $\models \varphi(m',m,b)$.

\begin{theorem}[\cite{CR}, Theorem 5.7]\label{thm:CR}
$T$ has NSOP$_1$ if and only if for every $M\models T$ and $b_0c_0\equiv_M b_1c_1$ such that $c_1\ind[u]_M c_0$, $c_0\ind[u]_M b_0$, and $c_1\ind[u]_M b_1$, there exists $b$ such that $bc_0\equiv_M b_0c_0\equiv_M b_1c_1 \equiv_M bc_1$.
\end{theorem}

For our purposes, the reader can take the independent $3$-amalgamation condition in Theorem~\ref{thm:CR} as the definition of NSOP$_1$. For the original definition and further discussion of this property, see~\cite{CR} or~\cite{DS}.

\begin{lemma}[\cite{CPZ}, Lemma 4.2]\label{lem:CPZ}
Let $a,b,c,d',d''$ be tuples and $F$ a finite set from a model $M\models \TCPZ$. Assume that $a$ and $c$ have only elements of $F$ in common ($a\ind[=]_F c$). If $d'a\equiv_F d'b \equiv_F d''b\equiv_F d''c$, then there exists $d$ such that $da\equiv_F d'a \equiv_F d''c\equiv_F dc$. 
\end{lemma}

\begin{corollary}\label{cor:CPZ}
$\TCPZ$ has NSOP$_1$.
\end{corollary}
\begin{proof}
Suppose we are given $M\models \TCPZ$ and $d'a \equiv_M d''c$ such that $c\ind[u]_M a$, $a\ind[u]_M d'$, and $c\ind[u]_M d''$. Let $p(x,y) = \mathrm{tp}(d'a/M) = \mathrm{tp}(d''c/M)$. To verify the condition in Theorem~\ref{thm:CR}, we need to show that $p(x,a)\cup p(x,c)$ is consistent. 

Suppose it is inconsistent. Then there is some finite subset $F\subseteq M$ such that letting $q(x,y) = \mathrm{tp}(d'a/F) = \mathrm{tp}(d''c/F)$, $q(x,a)\cup q(x,c)$ is inconsistent. Since $c\ind[u]_M a$, we certainly have $c\ind[=]_M a$. By increasing $F$, we may assume that $c\ind[=]_F a$. By countable categoricity, $q$ is isolated by a single formula $\theta(x,y)$ over $F$, and $\theta(d',y)\in \mathrm{tp}(a/Md')$, so by finite satisfiability there exists $b$ in $M$ satisfying $q(d',b)$. Since $d'\equiv_M d''$, we also have $\models q(d'',b)$. 

Now the assumptions of Lemma~\ref{lem:CPZ} are satisfied, and we can find $d$ satisfying $q(d,a)$ and $q(d,c)$, which contradicts inconsistency.
\end{proof}

Now we turn to pseudofiniteness of $\TCPZ$. The strategy is the same as in Section~\ref{sec:Tfeq}: filter the \Fraisse class $\KCPZ$ by bounding the number of equivalence classes, and expand to a class with disjoint $n$-amalgamation for all $n$ by naming the classes. 

\begin{theorem}\label{thm:cpzpseudfo}
$\TCPZ$ is pseudofinite.
\end{theorem}
\begin{proof}
For $n\geq 1$, let $K_n$ be the subclass of $\KCPZ$ consisting of those structures with the property that for all $k$, the equivalence relation $E_k$ has at most $n$ classes, in addition to the class of redundant tuples.

$K_n$ has the hereditary property, and the joint embedding property follows from the amalgamation property by taking $A$ to be the empty structure. For the disjoint amalgamation property, we wish to amalgamate embeddings $f\colon A\hookrightarrow B$ and $g\colon A\hookrightarrow C$ of structures in $K_n$. We specify a structure $D$ with domain $A\cup (B\setminus f[A])\cup (C\setminus g[A])$ into which $B$ and $C$ embed in the obvious way over $A$. Since the relations $E_k$ are independent, we can do this separately for each. Make sure to put all redundant $k$-tuples into the $E_k$-class reserved for them, number the $E_k$-classes which intersect $A$ nontrivially, then go on to number the classes which just appear in $B$ and $C$, and merge those classes which are assigned the same number, exactly as in Theorem~\ref{thm:tfeq}.

For any structure $A$ in $\KCPZ$, if $|A| = N$, then the number of $n$-tuples consisting of distinct elements from $A$ reaches its maximum of $N!$ when $n = N$. When $n>N$, every $n$-tuple from $A$ contains repeated elements. So the number of $E_n$-classes is bounded above by $N!+1$ for all $n$, and $A\in K_{N!+1}$. Hence $\Kfeq = \bigcup_{n = 1}^\infty K_n$. So $\KCPZ$ is a filtered \Fraisse class, and by Corollary~\ref{cor:filtered}, it suffices to show that each $T_{K_n}$ is pseudofinite.

Let $L_n'$ be the expanded language which includes, in addition to the relations $E_k$, $(n+1)$ $k$-ary relation symbols $C_k^0(\overline{x}), \dots, C_k^n(\overline{x})$ for each $k$. Let $K_n'$ be the class of finite $L_n'$-structures which are expansions of structures in $K_n$ such that for all $k$, each $E_k$-class is picked out by one of the $C_k^i$, with the class of redundant tuples picked out by $C^0_k$. 

We have $K_n = \{A\restriction L \mid A\in K_n'\}$, since every structure in $K_n$ can be expanded to one in $K_n'$ by labeling the classes for each equivalence relation. Suppose now that $(A,B)$ is a one-point extension in $K_n$, and $A'$ is an expansion of $A$ to a structure in $K_n'$. If any $k$-tuple involving the new element $b$ is part of a class which exists in $A$, we label it by the appropriate $C^i_k$. If adding the new element adds new $E_k$-classes, we simply label these classes by unused $C^j_k$ (by the bound $n$ on the number of classes, there will always be enough of the $C^j_k$). So $K_n'$ is a \Fraisse expansion of $K_n$.

It remains to show that $T_{K_n'}$ has disjoint $n$-amalgamation for all $n$. Suppose we have a coherent $\mathcal{P}^-([n])$-family of types. As noted before, the relations $E_k$ are independent, so we can handle them each separately. And the behavior of $E_k$ is entirely determined by the behavior of the relations $C^i_k$, so it suffices to set these. But the only restriction here is that every $k$-tuple should satisfy exactly one $C^i_k$, and it should be $C^0_k$ if and only if the tuple contains repeated elements. So to solve our amalgamation problem, we simply assign relations from the $C^i_k$ arbitrarily to those non-redundant $k$-tuples which are not already determined by the types in the family.

Hence $T_{K_n'}$ has disjoint $n$-amalgamation for all $n$, so it and its reduct $T_{K_n}$ are pseudofinite. 
\end{proof}

\begin{proposition}\label{prop:notuniform2}
There is a sentence $\varphi$ in $\TCPZ$ such that $$\lim_{n\to \infty} \mu_n(\{A\in\KCPZ(n)\mid A\models\varphi\}) = 0.$$
\end{proposition}
\begin{proof}
An example of such a sentence $\varphi$ is $\A{x}\A{y}\A{y'}\E{z} (E_1(x,z)\land E_2(y,y';x,z))$. This sentence says that for all $x$, the function $\rho_x$ mapping an element $z$ in the $E_1$-class of $x$ to the $E_2$-class of $xz$ is surjective onto the $E_2$-classes. $\varphi$ is in $\TCPZ$, since for any $A$ in $\Kfeq$ and elements $a$, $b$, and $b'$ in $A$, we can embed $A$ in a structure $B$ in $\Kfeq$ with an object $c$ such that $c$ is $E_1$-equivalent to $a$ and $(a,c)$ is $E_2$-equivalent to $(b,b')$. If $b = b'$, we must take $a = c$; otherwise, we can add a new element satisfying this condition. So $\varphi$ is implied by the relevant one-point extension axioms.

The measure $\mu_n$ amounts to picking an equivalence relation on the $k$-tuples of distinct elements from a set of size $n$ for each $k$ uniformly and independently. Since our sentence only involves $E_1$ and $E_2$, we just need consider the equivalence relations on $1$-tuples (there are $n$ of them) and the non-redundant $2$-tuples (of which there are $n^2-n$). Citing again the fact that the expected number of equivalence classes in a random equivalence relation grows asymptotically as $\frac{n}{\log(n)}(1 + o(1))$ \cite[Proposition VIII.8]{FS}, we see that with high probability there are more $E_2$-classes ($\frac{n^2}{\log(n^2-n)}(1+o(1))$) than the size of the average $E_1$-class ($\log(n)$), in which case the function $\rho_x$ is not surjective for all $x$, and the probability that $\varphi$ is satisfied converges to $0$.
\end{proof}

In this case, too, it would be interesting to know whether there is a zero-one law for the uniform measures.

\sloppy
\printbibliography

\end{document}